\documentclass{article}

\usepackage{arxiv}

\usepackage[utf8]{inputenc} 
\usepackage[T1]{fontenc}    
\usepackage{url}            
\usepackage{booktabs}       
\usepackage{amsfonts}       
\usepackage{nicefrac}       
\usepackage{microtype}      
\usepackage{lipsum}

\usepackage{relsize,balance,lipsum,bbm,enumerate,times,comment,color,graphicx,setspace,mathdots,mathrsfs,amssymb,latexsym,amsfonts,amsmath,cite,stmaryrd,caption,pgf,accents,mathtools,tabu,enumitem,hhline,array,epstopdf,nicefrac,amsthm,microtype,algorithmic,array,float,bm,url}
\newtheorem{thm}{Theorem}
\newtheorem{defn}{Definition}

\newtheorem{lem}{Lemma}
\newtheorem*{rem}{Remark}
\newtheorem{cor}{Corollary}
\newtheorem{assum}{Assumption}
\usepackage{graphicx}

\usepackage{mathtools} 

\usepackage[utf8]{inputenc}
\usepackage[english]{babel}
\newcommand{\an}[1]{{\textcolor{black}{#1}}}
\newcommand{\sa}[1]{{\textcolor{black}{#1}}}

\DeclareMathOperator*{\argmin}{arg\,min}
\usepackage{tikz}

\usetikzlibrary{positioning}
\definecolor{mygreen}{RGB}{153,255,153}
\definecolor{myorange}{RGB}{255,178,102}
\definecolor{myred}{RGB}{255,153,153}
\definecolor{myblue}{RGB}{153,204,255}

\title{On Non-Monotone Variational Inequalities}

\author{
    Sina Arefizadeh$^*$\\
    Dept. of Electrical and Computer Engineering\\
    Arizona State University\\
    Tempe, Arizona, USA\\
    \texttt{sarefiza@asu.edu}
\And
    Angelia~Nedi\'c\\
    Dept. of Electrical and Computer Engineering\\
    Arizona State University\\
    Tempe, Arizona, USA\\
    \texttt{Angelia.Nedich@asu.edu}
}

\theoremstyle{definition}
\def\argmin{\mathop {\rm argmin}}
\def\nat{{\rm nat}}

\begin{document}
\thanks{Corresponding author}
\thanks{ This work has been supported by the NSF award CCF 2106336.}
\maketitle
\begin{abstract}
In this paper, we \an{provide} some sufficient conditions for the existence of solutions to non-monotone Variational Inequalities (VIs) based on inverse mapping theory and degree theory. We have obtained several applicable sufficient conditions for this problem and have introduced a sufficient condition for the existence of a Minty solution. We have shown that the Korpelevich and Popov methods converge to a solution of a non-monotone VI, provided that a Minty solution exists. 
\end{abstract}

\section{Introduction}\label{Sec-introduction}
The Variational Inequality (VI) framework has been developed as a branch of mathematics to address practical challenges in finance, economics, social, pure, and applied sciences~\cite{noor2020new}.
Many solution concepts in game theory are tightly connected with solution concepts in variational inequalities~\cite{C20}. 
Recently, VIs have \an{attracted a lot of attention}, with main focus on designing algorithms for solving monotone VIs. 
Among these studies, the gradient methods started with work by Sibony~\cite{C9}, the proximal method by Martinent~\cite{C10}, \an{and extended} to the extra gradient method by Korpelevich~\cite{C4} and, later on, by the developments of forward-backward~\cite{C11}, mirror-prox~\cite{C12}, dual exploration~\cite{C13}, and hybrid proximal Korpelevich~\cite{C14} methods.
More recently, higher-order methods have been proposed and studied to determine their global \an{convergence properties and iteration complexities,} such as~\cite{C15}, \cite{C16}, and~\cite{C17}. 
For non-monotone VIs, there are algorithms that converge to a solution of a VI, when a Minty solution exists; see~\cite{C2} and the references therein. \an{However, the literature studying sufficient conditions for the existence of a solution to a non-monotone VI is scarce, and this paper aims to contribute to that literature.}

In this paper, we first study an unconstrained non-monotone VI using the inverse mapping theorem and derive some sufficient conditions guaranteeing the VI mapping has a zero, which is a solution to the unconstrained VI. In particular, we show that, when the VI mapping is continuously differentiable and its Jacobian is non-singular at all points where the mapping is non-zero, then the VI has a solution if the mapping has a closed range set. We also show that the same result holds if the closed range set condition is replaced with a requirement that the mapping is norm-coercive.
Next, we explore sufficient conditions for the existence of a solution to a constrained VI using the generalized Jacobian and the Clark inverse mapping theorem. Also, we establish some conditions for the existence of a solution to a constrained non-monotone VI based on the degree theory. Specifically, we have shown that a solution exists for the constrained VI when the non-monotone mapping is, in a sense, close to a $\xi$-monotone mapping. We also show the sub-sequential convergence of the Korpelevich and Popov method to solutions of a non-monotone VI, when a Minty solution exists. 
The main contributions of this paper are: \\
\noindent
(1)~Sufficient conditions for the existence of solutions and Minty solutions to {\it non-monotone} VIs constrained over (nonempty) closed and convex sets, potentially {\it unbounded} (Corollary~\ref{Cor-Alternative to main theorem 2}, Corollary~\ref{Cor-Alternative to main theorem 2-prime}).\\
\noindent
(2)~Existence of solutions results for a constrained non-monotone VI with the mapping that is uniformly close to a $\xi$-monotone mapping (Theorem~\ref{Thm-Degree theorem main conclusion}, Theorem~\ref{Thm-Minty sol}).
\\
\noindent
(3)~Sub-sequential convergence results for the Korpelevich and Popov methods for 
non-monotone VIs (Theorem~\ref{Thm-extragradient main variant}, Theorem~\ref{Thm-Popov main variant}). 

This paper builds upon our prior work in~\cite{arefizadeh2024non}. However, this paper has many new results and provides a more detailed and refined analysis compared to its preliminary version. In our prior work~\cite{arefizadeh2024non}, we have considered non-monotone VI and obtained some initial solutions in unconstrained VI and some results regarding the existence of a Minty solution in constrained VI, as well as convergence analysis for the Korpelevich method.
Compared to~\cite{arefizadeh2024non}, this current paper provides a more in-depth analysis of the results on the existence of solutions to constrained non-monotone VIs via norm-coercivity property of the mapping, and the use of the natural and the normal mappings associated with a given VI problem.
Additionally, 
new existence of solution results are provided based on the degree theory. Moreover, the convergence analysis of the Popov method is also new.

The organization of the rest of the paper is as follows: Section \ref{Sec-Notions & Terminology} introduces notions and terminology. 
Section~\ref{Sec-Main Results} presents the main results of the paper regarding the existence of solutions to non-monotone VIs. 
Section~\ref{Sec-Algorithm analysis} provides convergence results for the Korpelevich and Popov methods for non-monotone VIs, while
Section~\ref{Sec-conclusion} concludes the paper with a summary of contributions.


\section{Notions and Terminology} \label{Sec-Notions & Terminology}
In this section, we provide some definitions and terminologies about the variational inequalities. 

\begin{defn}
[Variational Inequality Problem \cite{facchinei2003finite}] Given a set $K\subseteq\mathbb{R}^m$ and a mapping $F:K\to\mathbb{R}^m$, the variational inequality problem, denoted by
VI$(K,F)$, consists of determining
   a point $x^*\in K$ such that 
   $$\langle F(x^*),x-x^*\rangle\geq 0\qquad\hbox{for all $x\in K$}.$$ 
   \end{defn}
A point $x^*$ satisfying the preceding inequality is a (strong) solution to the variational inequality problem VI$(K,F)$. The set of all solutions is denoted by SOL$(K,F)$.

Another concept of a solution to a VI($K,F)$ exists, known as weak or Minty solution, defined as follows. 
   \sa{\begin{defn}[Minty Solution Concepts \cite{crespi2005existence}]\label{Def-Strong-Minty}
   Given a VI$(K,F)$ and an $\eta\in\mathbb{R}$, a point $x^*\in K$  such that
   $$\langle F(x),x-x^*\rangle\geq \eta\|x-x^*\|^2\qquad\hbox{for all $x\in K$},$$ is a strong Minty solution to VI$(K,F)$ if $\eta>0$, a Minty solution if $\eta=0$, and a weak Minty solution if $\eta<0$.
\end{defn}}
For the most part, we will deal with a Minty solution ($\eta=0$).The set of all Minty solutions to VI$(K,F)$ is denoted by MSOL$(K,F)$.

A well-studied class of variational inequalities is the class of monotone VIs. The following definition introduces different types of monotonicity properties for a mapping.  
\begin{defn}[Monotone Mappings \cite{C2}]
   Given a set $K\subseteq\mathbb{R}^m$ and a mapping $F:K\to\mathbb{R}^m$, the mapping $F$ is $\xi$-monotone if for some $\xi>1$ there exists a constant $c_F>0$ such that for all $x,y\in K$, \[\langle F(x)-F(y),x-y\rangle\geq c_F\|x-y\|^\xi.\]
   When $\xi=2$, 
   the mapping $F$ is strongly monotone. 
   Moreover, the mapping $F$ is monotone if the preceding relation holds with $c_F=0$, i.e., for all $x,y\in K$, \[\langle F(x)-F(y),x-y\rangle\geq 0.\]
\end{defn}

Given a VI$(K,F)$, we say that 
the VI is 
$\xi$-monotone, strongly monotone, or monotone when the mapping $F$ has such monotonicity properties, respectively. 

The following result holds for solutions and Minty solutions of a VI.
\begin{lem}[Lemma 2.2 - Minty's Lemma \cite{C2}]\label{Lem-Minty Lemma}
    Let $K \subseteq \mathbb{R}^m$ be a non-empty closed set
    and let $F:K \to\mathbb{R}^m$ be a mapping. The following statements hold:
    \begin{itemize}
        \item[(a)]
        If $F$ is continuous and the set $K$ is convex, then
    every Minty solution to VI$(K,F)$ is also a solution to VI$(K,F)$,  i.e., 
    \[{\rm MSOL}(K,F)\subseteq {\rm SOL}(K,F).\]
    \item[(b)] If $F$ is monotone, then every solution to VI$(K,F)$ is also a Minty solution to VI$(K,F)$, i.e.,
    \[{\rm SOL}(K,F)\subseteq {\rm MSOL}(K,F).\]
    \end{itemize}
\end{lem}
\begin{proof}  
(a) Let us consider an $x^*\in {\rm MSOL}(K,F)$. Hence, \begin{equation}\label{eq-minty lemma 1}
\langle F(x),x-x^*\rangle\geq 0\qquad\hbox{for all $x\in K$}.
   \end{equation}
For any arbitrary $x\in K$, consider the point $v=x^*+t(x-x^*)$ where $t\in (0,1]$. Note that $v\in K$ since $x,x^*\in K$ and $K$ is convex. Thus, by using $v\in K$,
from~\eqref{eq-minty lemma 1} we have
\begin{equation*}
    \langle F(x^*+t(x-x^*)),t(x-x^*)\rangle\geq 0\quad\hbox{for all $x\in K,t\in (0,1]$}.
    \end{equation*}
 Since $t>0$, it follows that
\begin{equation*}
    \langle F(x^*+t(x-x^*)),x-x^*\rangle\geq 0\qquad\hbox{for all $x\in K,t\in (0,1]$}.
    \end{equation*}
Letting $t$ go to zero and using the continuity of $F$ we have
Hence, $x^*\in {\rm SOL}(K,F)$, implying that
\[{\rm MSOL}(K,F)\subseteq {\rm SOL}(K,F).\]
(b) Let $x^*$ be a solution to VI$(K,F)$.
By the monotonicity of $F$, we have
\begin{equation}\label{eq-minty lemma 5}
\langle F(x)-F(x^*),x-x^*\rangle\geq 0.
\end{equation}
Since $x^*\in {\rm SOL}(K,F)$ it follows that $\langle F(x^*),x-x^*\rangle\geq 0$. Therefore, from~\eqref{eq-minty lemma 5} it follows that $\langle F(x),x-x^*\rangle\geq 0$, thus implying that $x^*\in {\rm MSOL}(K,F)$. 

\end{proof}

Combining parts (a) and (b) of Lemma~\ref{Lem-Minty Lemma}, we see that 
for a continuous and monotone mapping $F$
we have
\[{\rm MSOL}(K,F)= {\rm SOL}(K,F).\] 
This result has been shown in Lemma~1.5 of \cite{KSbook}, where both monotonicity and continuity of the mapping $F$ are assumed. Our Lemma~\ref{Lem-Minty Lemma} considers these properties separately to gain a deeper insight into the role of these properties in the relations among the solutions and Minty solutions of a VI$(K,F)$.

\section{Main Results} \label{Sec-Main Results}
\def\Rem{\mathbb{R}^m}

In this section, we develop some sufficient conditions for the existence of solutions to unconstrained and constrained VI problems, separately. For the unconstrained VIs, we consider differentiable mappings. For the constrained VIs, we consider the mappings that are Lipschitz continuous but not necessarily differentiable.

\subsection{Unconstrained VI}\label{ss-dif-map}
In this section, we consider an
unconstrained VI$(\Rem,F)$ with 
a differentiable mapping $F:\Rem\to\Rem$.
We use $\nabla F(\cdot)$ for the Jacobian of a mapping $F(\cdot)$ and   $|\nabla F(x)|$ for the determinant of the Jacobian $\nabla F(x)$.
Next, we state the Inverse Mapping Theorem, which we use in the subsequent development.

\begin{thm}[Inverse Mapping Theorem \cite{lebl2018introduction}]\label{Thm-Inverse function}
    Given a vector $\textbf{a}$, let $F(\cdot)$ be a continuously differentiable mapping on some open set containing the vector $\textbf{a}$. Suppose that $|\nabla F(\textbf{a})|\neq 0$. Then, there is an open set $V$ containing the vector $\textbf{a}$ and an open set $W$ containing the vector $F(\textbf{a})$ such that $F:V\rightarrow W$ has a continuous inverse mapping $F^{-1}:W\rightarrow V$, which is continuously differentiable on~$W$.  
\end{thm}

The following result provides sufficient conditions guaranteeing the existence of a solution to VI$(\mathbb{R}^m,F)$.
\begin{thm} \label{Thm-main theorem}
    Let $F:\mathbb{R}^m\to\mathbb{R}^m$ be a continuously differentiable mapping with a closed range set. i.e., the set $F(\Rem)=\{F(x)\mid x\in\Rem\}$ is closed. Also, assume that 
    $|\nabla{F(x)}|\neq 0$
    for every $x\in \mathbb{R}^m$ satisfying $F(x)\neq 0$. Then, the VI$(\mathbb{R}^m,F)$ has a solution.
\end{thm}
\begin{proof}
    Let $b=\inf_{x\in \mathbb{R}^m} \|F(x)\|$ and let 
    $\{x^k\}$ be a sequence such that $\lim_{k\rightarrow \infty}\|F(x^k)\|=b$.
    Suppose that $b>0$.
Since $\|F(x^k)\|\to b$, it follows that the sequence $\{F(x^k)\}$ is bounded and has a convergent sub-sequence $\{F(x^{k_i})\}$, i.e., $\lim_{i\rightarrow \infty}F(x^{k_i})=\bar{F}$ and $\|\bar{F}\|=b$. 
Since the set $F(\mathbb{R}^m)$ is closed, it follows that $\bar F\in F(\mathbb{R}^m)$. Hence, there is some $\bar{x}\in \mathbb{R}^m$ such that $F(\bar{x})=\bar{F}$. Since $\|\bar{F}\|=b$ and $b>0$, we have that  $F(\bar{x})\ne 0$. By the assumption of the theorem, it follows that $|\nabla{F(\bar{x})}|\neq 0$.

Let $B_r(x)$ be an open ball centered at a point $x\in \mathbb{R}^m$ with a radius $r>0$. Since $|\nabla{F(\bar x)}|\neq 0$,
    by Inverse Mapping Theorem~\ref{Thm-Inverse function}, there exist open balls $B_r(\bar x)$ and $B_{r'}(F(\bar x))$, and a local inverse mapping $F^{-1}_{\bar x}(\cdot):B_{r'}(F(\bar x))\to B_r(\bar x)$ such that $F^{-1}_{\bar x}(v)=u$ for all $v\in B_{r'}\left(F(\bar x)\right)$ and $u\in B_{r}(\bar x)$, where $F(u)=v$.
    Thus, there exists $\alpha\in(0,1)$ such that $$(1-\alpha)F(\bar{x})\in B_{r'}\left(F(\bar{x})\right).$$ 
    Then, by the Inverse Mapping Theorem, we have that $F^{-1}_{\bar{x}}\left((1-\alpha)F(\bar{x})\right)=z$ for some $z\in B_{r}(\bar{x})$, such that $$F(z)=F\left(F^{-1}_{\bar{x}}\left((1-\alpha)F(\bar{x})\right)\right)=(1-\alpha)F(\bar{x}).$$ Hence, $0\leq\|F(z)\|< \|F(\bar{x})\|=b$ with $z\in \mathbb{R}^m$, which cannot hold since $b=\inf_{x\in \mathbb{R}^m} \|F(x)\|$. Therefore, we cannot have $b>0$, implying that  $b=0$ and $F(\bar x)=0$, which means that $\bar{x}\in {\rm SOL} (\mathbb{R}^m,F)$.
    
\end{proof}

The next two lemmas provide sufficient conditions for non-singularity of the Jacobian, $|\nabla{F(x)}| \neq 0$. In what follows, we use $[m]$ to denote the set $\{1,2,\ldots,m\}$ for an integer $m\ge1$.
Also, we use $\nabla_{x_i}f(x)$ to denote the partial derivative of a function $f$ with respect to the variable $x_i$.

\begin{lem}[Weak Coupling Condition]\label{Lem-Weak coupling}
Consider a mapping $F:\mathbb{R}^m\to\mathbb{R}^m$ given by
$F=(F_1,F_2,\ldots, F_m)$, where $F_i:\mathbb{R}^m\to\mathbb{R}$ is the $i$-th component of the mapping $F$ for all $i\in[m]$.
At a given point $x\in\Rem$, if  
\[\left|\nabla_{x_i} F_i(x)\right|>\sum_{\substack{j=1 \\ j \neq i}}^m\left|\nabla_{x_j}F_i(x)\right|
\qquad\hbox{for every $i\in [m]$},\]
 then $|\nabla{F(x)}| \neq 0$.
\end{lem}
\begin{proof}   
Letting $M_{ij}$ be the $ij$-th element of an $m\times m$ matrix $M$, for the matrix $M$,  we define the Gershgorin disk $C_i\left(M_{ii}, \sum_{\substack{j=1 \\ j \neq i}}^m\left|M_{ij}\right|\right)\subset \mathbb{C}$  as a disk in the complex plane $\mathbb{C}$ centered at $M_{ii}$ with the radius of $\sum_{\substack{j=1, j \neq i}}^m\left|M_{ij}\right|$. By Gershgorin Circle Theorem \cite{bell1965gershgorin}, every eigenvalue of the Jacobian $\nabla F(x)$ lies in one of the Gershgorin disks $C_i\left(\nabla_{x_i} F_i(x), \sum_{\substack{j=1 \\ j \neq i}}^m\left|\nabla_{x_j} F_i(x)\right|\right)$, $i \in[m]$. Under the weak coupling condition, it follows that  $0$ is not in any of the Gershgorin disks. Thus, $0$ is not an eigenvalue of $\nabla F(x)$, implying that $|\nabla{F(x)}| \neq 0$.

\end{proof} 



We use the following lemma to show that a class of non-monotone unconstrained VIs has a solution.

\begin{lem}\label{Lem-Alternative to be non singular}
Let $F:\mathbb{R}^m\to\mathbb{R}^m.$ Then, for any $x\in\mathbb{R}^m$, we have    $\left|\nabla{F(x)}\right| \neq 0$ if and only if the matrix $\nabla{F(x)} \nabla{F(x)}^T$ is positive definite.
 \end{lem}
\begin{proof} We have that
$|\nabla{F(x)}|^2=|\nabla{F(x)} \nabla{F(x)}^T|$. Thus, the determinant $|\nabla{F(x)}|^2$ is zero if and only if $|\nabla{F(x)} \nabla{F(x)}^T|=0$. So, alternatively, we can consider the matrix $\nabla{F(x)}\nabla{F(x)}^T$, which is symmetric. Hence, all the eigenvalues of $\nabla{F(x)} \nabla{F(x)}^T$ are real and the matrix is positive semi-definite.  Therefore, $\left|\nabla{F(x)}\right| \neq 0$ is equivalent to $\nabla{F(x)} \nabla{F(x)}^T$ being positive definite.

 \end{proof}

 As a consequence of Theorem~\ref{Thm-main theorem} and Lemma~\ref{Lem-Alternative to be non singular}, we have the following result.

\begin{cor}\label{Cor-Orthogonal map}
 Let $F:\mathbb{R}^m\to\mathbb{R}^m$ be a continuously differentiable mapping with the closed range set $F(\Rem)$. Assume that $\nabla{F(x)}^{-1}=\nabla{F(x)}^T$
 for all $x\in\mathbb{R}^m$ where $\nabla{F(x)}\neq 0$.  
Then, the
VI$\left(\mathbb{R}^m, F\right)$ has a solution.
\end{cor}
\begin{proof}
    For all $x\in\mathbb{R}^m$ satisfying $\nabla{F(x)}\neq 0$, we have  $F(x)\nabla F(x)^T=I$, where $I$ is the identity mapping. By Lemma~\ref{Lem-Alternative to be non singular} we have that $|\nabla F(x)|\neq 0$,  and by Theorem~\ref{Thm-main theorem}, VI$\left(\mathbb{R}^m, F\right)$ has a solution.
    
\end{proof}

The condition of Theorem~\ref{Thm-main theorem} requiring the closed range set $F(\Rem)$ is satisfied, for example, when the mapping $F$ is {\it closed}, which is defined as follows~\cite{C1}:
{\it A mapping $F:\mathbb{R}^m\rightarrow \mathbb{R}^m$ is closed if for every closed set $C\subseteq \mathbb{R}^m$, the image set $F(C)$ is closed, where $F(C)=\{F(x)\mid x\in C\}$.
}

Next, we consider a norm-coercive mapping.
\begin{defn}[Norm-coercive\,Mapping\,\cite{facchinei2003finite},\,p.~134]
    A mapping $F:\mathbb{R}^m\to \mathbb{R}^m$ is norm-coercive if it satisfies $\lim_{\|x\|\to\infty} \|F(x)\|=+\infty$. 
\end{defn}

Theorem~\ref{Thm-main theorem} also holds when the closed range set condition is replaced with the requirement that the mapping is norm-coercive, as given
in the following assumption.

\begin{assum}\label{Assum-Semi coercivity}
A mapping $F:\mathbb{R}^m\to\mathbb{R}^m$ is norm-coercive.
 \end{assum}


Now, we have the following result showing another set of conditions that are sufficient for the existence of a solution to a VI$(\Rem,F)$.

\begin{thm} \label{Thm-main theorem-prime}
    Let $F:\mathbb{R}^m\rightarrow\mathbb{R}^m$ be a continuously differentiable mapping. Also, let Assumption~\ref{Assum-Semi coercivity} hold and assume that 
    $|\nabla{F(x)}|\neq 0$
    for every $x\in \mathbb{R}^m$ satisfying $F(x)\neq 0$. Then,  the $VI(\mathbb{R}^m,F)$ has a solution.
\end{thm}
\begin{proof}
Let $b=\inf_{x\in \mathbb{R}^m} \|F(x)\|$ and  let $\{x^k\}$ be a sequence such that $\lim_{k\rightarrow \infty}\|F(x^k)\|=b$. 
Thus, the sequence $\{F(x^k)\}$ is bounded and has a convergent sub-sequence $\{F(x^{k_i})\}$ with $\lim_{i\rightarrow \infty}F(x^{k_i})=\bar{F}$, where $\|\bar{F}\|=b$. Moreover, by Assumption~\ref{Assum-Semi coercivity}, the sequence $\{x^{k_i}\}$ is also bounded and, consequently, has a sub-sequence converging to some $\bar{x}$. 
Along this sub-sequence, the mapping values $F(x^{k_i})$ are converging to $\bar{F}$. Without loss of generality, we may assume that $\lim_{i\to\infty}x^{k_i}=\bar x$ and  $\lim_{i\to\infty}F(x^{k_i})=\bar F$. Since $F$ is continuous, it follows that  
$F(\bar x)=\bar F$, where $\|\bar F\|=b$.
To arrive at a contradiction,  assume that $b>0$. By the Inverse Mapping Theorem (Theorem~\ref{Thm-Inverse function}), since $|\nabla{F(\bar x)}|\neq 0$, there are open balls $B_r(\bar x)$ and $B_{r'}\left(F(\bar x)\right)$, and a locally invertible mapping $F^{-1}_{\bar x}(v)$ (for the mapping $F$) such that $F^{-1}_{\bar x}(v)=u$ for all $v\in B_{r'}\left(F(\bar x)\right)$ and $u\in B_{r}(\bar x)$, where $F(u)=v$.
From now onward, the proof follows the same line of analysis as that of Theorem~\ref{Thm-main theorem}, leading to a contradiction that   
$0\leq\|F(z)\|< \|F(\bar{x})\|=b$ for some $z\in \mathbb{R}^m$. Therefore,  we must have $b=0$ and $F(\bar x)=\bar{F}=0$, implying that $\bar{x}$ is a solution of the VI$(\mathbb{R}^m,F)$.

\end{proof}

\subsection{Constrained VI}\label{ss-const}
In this section, we \an{investigate sufficient conditions for the existence of solutions to} constrained variational inequality problems using two different approaches. Firstly, we extend the results for an unconstrained VI of Section~\ref{ss-dif-map} to the constrained case based on the inverse mapping theory. Secondly, we establish the existence results based on the degree theory.
\subsubsection{Inverse Mapping Based Approach}
We consider 
a constrained VI problem VI$(K,F)$, where $K$ is a nonempty, closed, and convex set, not necessarily bounded, and we establish the existence of solution results that parallel the unconstrained VI case in the sense that we use the inverse mapping theory but in a non-differentiable setting. To proceed with this, following a standard approach~\cite{facchinei2003finite},
we use an alternative formulation of the VI$(K,F)$ problem  related to finding a zero of the natural mapping associated with the VI$(K,F)$, defined as follows:
\[F_K^{\rm nat}(v)=v-\Pi_K\left[v-F(v)\right]
\qquad\hbox{for all $v\in\mathbb{R}^m$},\] where $F:\Rem\to\mathbb{R}^m$ and $\Pi_K[\cdot]$ is the Euclidean projection on the closed convex set $K\subseteq\mathbb{R}^m$, i.e., $\Pi_K[z]=\argmin_{x\in K}\|x-z\|^2$.
The following theorem relates the solutions of the VI$(K,F)$ problem with the zeros of its associated natural mapping $F^{\nat}_K$.

\begin{thm}[Proposition 1.5.8 \cite{facchinei2003finite}]\label{Thm-auxiliary-natural map solution}
    Let $K\subseteq\Rem$ be a nonempty, closed, and convex set, and let
    $F:\an{\Rem}\rightarrow \mathbb{R}^m$ be a mapping. Then, we have
    \begin{equation}
        [x^*\in {\rm SOL}(K,F)] \iff [F_K^{\nat}(x^*)=0].
    \end{equation}
\end{thm}

The natural mapping $F^{\nat}_K$ need not be differentiable even when $F$ is differentiable.
 Thus, to obtain an alternative to 
 Theorem~\ref{Thm-main theorem} 
 where the mapping differentiability assumption is relaxed, we assume that \an{$F:\Rem\to\Rem$ is Lipschitz continuous i.e., for some $L>0$, 
    \[\|F(x)-F(y)\|\le L\|x-y\|\quad\hbox{for all }x,y\in \Rem.\]
    When $F:K\to\Rem$, with $K\subseteq\Rem$, then we say that $F$ is Lipschitz continuous on $K$, if the preceding relation holds restricted to the set $K$, i.e., for all $x,y\in K$.}
 
 The Lipschitz continuity of $F$ will allow us to use a variant of the Inverse Mapping
 Theorem~\ref{Thm-Inverse function} applicable to mappings that are not necessarily differentiable. To this end, we first define several concepts in the sequel.
\begin{defn}[Generalized\,Jacobian, Def.\,1\,\cite{clarke1976inverse}] \label{Def-Generalized Jacobian}
The generalized Jacobian of a mapping $F$ at a point $x\in \mathbb{R}^m$, denoted by $\partial F(x)$, is the convex hull of all matrices $M$ of the following form: 
\begin{equation}
\label{eq-gen-jac}
    M=\lim_{k\rightarrow \infty} \nabla F(x^k),
\end{equation}
where $\lim_{k\to\infty}x^k=x$ and $F$ is differentiable at $x^k$ for all~$k$.
\end{defn}

We use the following result for the generalized Jacobian.

\begin{thm}[$\partial F(\cdot)$ Properties, Proposition 1 \cite{clarke1976inverse}]\label{Thm-Proposition 1 of Clark paper}
    Let mapping $F$ be Lipschitz continuous in a neighborhood of a point $x\in \mathbb{R}^m$. Then, the generalized Jacobian $\partial F(x)$ is a nonempty, compact, and convex set in the space $\mathcal{M}$ of all square matrices of dimension $m$ topologized with norm $\|M\|=\max_{1\le i,j\le m}|m_{i,j}|$.
\end{thm}

We also use the notion of a generalized Jacobian of maximal rank, defined as follows.
\begin{defn}[Definition 2 \cite{clarke1976inverse}] \label{Def-fullrank Generalized Jacobian}
A generalized Jacobian
$\partial F(x)$ is said to be of maximal rank if every matrix $M$ in the definition of $\partial F(x)$ (see~\eqref{eq-gen-jac}) has a full rank.
\end{defn}

The following theorem, known as the Clark Inverse Mapping Theorem, is the key to extending our Theorem~\ref{Thm-main theorem} to the case of a constrained VI$(K,F)$.

\begin{thm}[Clark Inverse Theorem 1, \cite{clarke1976inverse}] \label{Thm-Clark nverse mapping theorem}
Let $\Phi:\mathbb{R}^m\to\mathbb{R}^m$ be a mapping.
    Let $\partial \Phi(x_0)$ be of a maximal rank for some $x_0\in\mathbb{R}^m$. Then, there exist neighborhoods $U$ and $V$ of $x_0$ and  $\Phi(x_0)$, respectively, and a Lipschitz continuous mapping $G:V\to \mathbb{R}^m$ such that 
    \begin{align*}
        G\left(\Phi(u)\right)&=u\qquad \hbox{for every $u\in U$},\\
        \Phi\left(G(v)\right)&=v\qquad\hbox{for every $v\in V$}.
    \end{align*}
\end{thm}

The following corollary gives a sufficient condition to have a solution for VI$(K,F)$.

\begin{cor}\label{Cor-Alternative to main theorem}
    Let set $K\subseteq\mathbb{R}^m$ be nonempty closed convex and let $F:\an{\Rem}\to\Rem$ be Lipschitz continuous. 
    Assume that $F_K^{\nat}:\Rem\to\Rem$ has a closed range set. Also, assume that the generalized Jacobian $\partial F_K^{\nat}(x)$ has a maximal rank for every $x\in \mathbb{R}^m$ where $F_K^{\nat}(x)\neq 0$. Then, the VI$(K,F)$ has a solution.
\end{cor}
\begin{proof}
For the natural mapping $F_K^{\nat}$ we have \an{for all $u,v\in\Rem$},
\begin{align*}
    \|F_K^\nat(u) &-F_K^\nat(v)\|\le \|u-v\|+\|\Pi_K[u-F(u)]-\Pi_K[v-F(v)]\|.
\end{align*}
Since the set $K$ is nonempty, closed, and convex, 
the following non-expansiveness property of the projection mapping holds 
\[\|\Pi_K[u]-\Pi_K[v]\|\le \|u-v\| \qquad\hbox{for all $u,v\in \Rem$.}\]
Thus, it follows that \an{for all $u,v\in\Rem$},
\begin{align*}
    \|F_K^\nat(u)-F_K^\nat(v)\|\le 2\|u-v\| +\|F(u)-F(v)\|\le (2+L)\|u-v\|,
\end{align*}
where we use Lipschitz continuity of $F$.  Thus,  $F_K^{\nat}$ is Lipschitz continuous with constant $2+L$
and, by Theorem~\ref{Thm-Proposition 1 of Clark paper}, 
the generalized Jacobian $\partial F_K^{\nat}(x)$ is nonempty.
The rest of 
    the proof follows along the same line of analysis as that of the proof of Theorem~\ref{Thm-main theorem}, where we replace $F_K^{\nat}(x)$  with $F(x)$ and use the Clark Inverse Mapping theorem instead of  the Inverse Mapping theorem. Also, we use the relation for the zeros of the natural mapping $F_K^{\nat}(\cdot)$ and the solutions of the VI$(K,F)$ given in Theorem~\ref{Thm-auxiliary-natural map solution}.
    
\end{proof}

The following result is an analog of Theorem~\ref{Thm-main theorem-prime} for the case when the mapping $F$ is not continuously differentiable but Lipschitz continuous instead.

\begin{cor}\label{Cor-Alternative to main theorem 2}
   Let $K\subseteq\Rem$ be closed convex set and $F:\an{\Rem}\to\Rem$ be Lipschitz continuous mapping. Let the mapping $F_K^{\rm nat}$ be norm-coercive (Assumption~\ref{Assum-Semi coercivity}) 
   and let the generalized Jacobian {$\partial F_K^{\rm nat}(x)$ be of maximal rank for every $x\in \mathbb{R}^m$ satisfying $F_K^{\rm nat}(x)\neq 0$.} Then, the VI$(K,F)$ has a solution.
\end{cor}

\begin{proof}
  Proof follows by following similar steps to those in the proof of Theorem~\ref{Thm-main theorem-prime}, where we consider $F_K^{\rm nat}(x)$ instead of $F(x)$, we employ the Clark Inverse Mapping theorem instead of the Inverse Mapping theorem, and use the connection between a zero of the natural mapping and a solution of the VI$(K,F)$ (Theorem~\ref{Thm-auxiliary-natural map solution}). 
  
\end{proof}

Generally, guaranteeing that $F_K^{\rm nat}$ is norm-coercive \an{might be involved depending on the mapping $F$ and the structure of the set $K$.}

Another way to study the existence of solutions is viable through the use of the normal map associated with the VI$(K,F)$, defined as follows:
\[F_K^{\rm nor}(v)=v-\Pi_K\left[v\right]+F(\Pi_K\left[v\right])
\qquad\hbox{for all $v\in\mathbb{R}^m$}.\]
\an{Unlike the natural mapping $F_K^{\rm nat}$, the normal mapping $F_K^{\rm nor}$ is well defined when the mapping $F$ is defined on the set $K$ instead of the entire space, i.e., $F:K\to\Rem$.}
The following is a known necessary and sufficient condition for the existence of a solution to a VI$(K,F)$ via the normal map $F_K^{\rm nor}$. 
\begin{thm}[Proposition 1.5.9 \cite{facchinei2003finite}]\label{Thm-auxiliary-normal map solution}
    Let $K\subseteq \mathbb{R}^m$ be a closed convex set and $F:K\to \mathbb{R}^m$ be a mapping. Then, we have $x^*\in{\rm SOL}(K,F)$ if and only if there is a vector $v\in\Rem$ such that $x^*=\Pi_K(v)$ and $F_K^{\rm nor}(v)=0$.
\end{thm}

The following lemma establishes some properties of the normal mapping $F_K^{\rm nor}$ associated with the VI$(K,F)$, when $F$ is norm-coercive.
\begin{lem}\label{Lem-Relation between mapping F and natural map}
   Let the set $K\subseteq \mathbb{R}^m$ be nonempty, closed, and convex. Let $F:K\to\mathbb{R}^m$ \an{be norm-coercive over the set $K$, i.e., $\lim_{\|x\|\to\infty, x\in K} \|F(x)\|=+\infty$.}
   Then, 
   \an{for any sequence $\{x^k\}\subset\Rem$ with $\lim_{k\to\infty}\|x^k\|=+\infty$, 
   either
   $\lim_{k\to\infty}\|F_K^{\rm nor}(x^k)\|= \infty$ or there exists a subsequence 
   $\{x^{k_i}\}\subset\{x^k\}$ such that} 
\[\lim_{i\to \infty}\frac{\langle x^{k_i}-\Pi_K[x^{k_i}],F(\Pi_K[x^{k_i}]) \rangle}{\|x^{k_i}-\Pi_K[x^{k_i}]\|\,\|F(\Pi_K[x^{k_i}])\|}=-1.\]
\end{lem}
\begin{proof}
For any sequence $\{x^k\}\subset \mathbb{R}^m$, and for all $k$,
\begin{align}\label{eq-Lem relation 2}
   \|F_K^{\rm nor}(x^k)\|^2
   =&\|x^k-\Pi_K[x^k]\|^2+\|F(\Pi_K[x^k])\|^2+2\langle x^k-\Pi_K[x^k], F(\Pi_K[x^k])\rangle\cr
=&\|x^k-\Pi_K[x^k]\|^2+\|F(\Pi_K[x^k])\|^2+2\|x^k-\Pi_K[x^k]\|\, \|F(\Pi_K[x^k])\rangle\| 
   \cos \theta^k, 
\end{align}
where $\theta^k$ is the angle between $x^k-\Pi_K[x^k]$ and $F(\Pi_K[x^k])$.
Let $\mathcal{K}$ be the index set such that $x^k\notin K$ for all $k\in\mathcal{K}$. Note that $F_K^{\rm nor}(x)=F(x)$ for all $x$. Thus, if the subsequence $\{x_k\mid k\not \in\mathcal{K}\}$ (which is contained in the set $K$) is infinite, then by the norm-coercivity of $F$ over the set $K$, we have $\lim_{k\to\infty, k\notin\mathcal{K}}\|F_K^{\rm nor}(x^k)\|=+\infty.$
Consider now the subsequence $\{x_k\mid k\in\mathcal{K}\}$. If this sequence is finite, then we are done. So assume that $\{x_k\mid k\in\mathcal{K}\}$ is infinite. \an{Depending whether $\{\|\Pi_K[x^k]\|\mid k\in\mathcal{K}\}$ is bounded or not, we consider two cases, respectively:}\\
\noindent
{\it Case $\{\|\Pi_K[x^k]\|\mid k\in\mathcal{K}\}$ is bounded}:  Dividing relation~\eqref{eq-Lem relation 2} by $\|x^k-\Pi_K[x^k]\|^2\ne0$ for $k\in\mathcal{K}$, 
we have 
\[\lim_{k\to \infty, k\in\mathcal{K}}\frac{\|F_K^{\rm nor}(x^k)\|^2}{\|x^k-\Pi_K[x^k]\|^2}=1.\] 
Since $\|x^k\|\to+\infty$ and $\{\|\Pi_K[x^k]\|\mid k\in\mathcal{K}\}$  is bounded, it follows that 
\[\lim_{k\to \infty,k\in\mathcal{K}}\|F_K^{\rm nor}(x^k)\|=\infty.\]

\noindent
{\it Case $\{\|\Pi_K[x^k]\|\mid k\in\mathcal{K}\}$ is unbounded}:  By relation~\eqref{eq-Lem relation 2}, we have for all $k\in\mathcal{K}$,
\begin{align}\label{eq-Lem relation 3}
   \|F_K^{\rm nor}(x^k)\|^2=&\left(\|x^k-\Pi_K[x^k]\|+\|F(\Pi_K[x^k])\| \cos \theta^k\right)^2+(1-\cos^2 \theta^k)\|F(\Pi_K[x^k])\|^2. 
\end{align} 
Hence,
\begin{align*}
\liminf_{k\to\infty\atop k\in \mathcal{K}}\|F_K^{\rm nor}(x^k)\|^2
   \ge 
   \liminf_{k\to\infty
   \atop k\in \mathcal{K}}
   (1-\cos^2 \theta^k)\|F(\Pi_K[x^k])\|^2.
   \end{align*}

If $\liminf_{k\to\infty, k\in \mathcal{K}}
   (1-\cos^2 \theta^k)\ne 0$, since $\{\|\Pi_K[x^k]\|\mid k\in\mathcal{K}\}$ is unbounded, by the norm-coercivity of $F$ over the set $K$, it follows that 
   \[\liminf_{k\to\infty\atop k\in \mathcal{K}}\|F_K^{\rm nor}(x^k)\|^2
   \ge +\infty.\]
If there is a subsequence $\{x^{k_j}\mid k_j\in\mathcal{K}\}$ along which $\cos \theta^{k_j}$ tends to 1, then 
the second term in \eqref{eq-Lem relation 3} disappears, but the first term tends to $+\infty$. 
Finally, if there is a subsequence $\{x^{k_i}\mid k_i\in\mathcal{K}\}$
along which $\cos \theta^{k_i}$ tends to -1, 
then we have a subsequence $\{x^{k_i}\}$ such that 
    \[\lim_{i\to \infty}\frac{\langle x^{k_i}-\Pi_K[x^{k_i}],F(\Pi_K[x^{k_i}]) \rangle}{\|x^{k_i}-\Pi_K[x^{k_i}]\|\|F(\Pi_K[x^{k_i}])\|}=-1,\] 
    which completes the proof.

\end{proof}

The following result is a consequence of Lemma~\ref{Lem-Relation between mapping F and natural map}.
\begin{cor}\label{Cor-Alternative to main theorem 2-prime}
   \sa{Let the set $K$ be closed convex, and let $F$ be Lipschitz continuous on $K$ and norm-coercive over the set $K$. Suppose there is no sequence $\{x^k\}$ with $\lim_{k\to \infty}\|x^k\|=\infty$ such that
   \begin{equation} \label{eq-Assumption-cor}
       \lim_{k\to \infty}\frac{\langle x^k-\Pi_K[x^k],F(\Pi_K[x^k]) \rangle}{\|x^k-\Pi_K[x^k]\|\|F(\Pi_K[x^k])\|}=-1.
   \end{equation}
   Moreover, let generalized Jacobian $\partial F_K^{\rm nor}(x)$ be of maximal rank for every $x\in \mathbb{R}^m$ where $F_K^{\rm nor}(x)\neq 0$. Then VI$(K,F)$ has a solution.}
\end{cor}
\begin{proof}
    The norm-coercivity of $F_K^{nor}$ is guaranteed by Lemma~\ref{Lem-Relation between mapping F and natural map}. The rest of the proof follows from Theorem~\ref{Thm-auxiliary-normal map solution}, and the proof line is similar to that of Corollary~\ref{Cor-Alternative to main theorem 2}, where we use $F_K^{\rm nor}(x)$ instead of $F_K^{\rm nat}(x)$.
    
\end{proof}

The non-existence of a sequence $\{x^k\}$ with $\lim_{k\to \infty}\|x^k\|=\infty$ and satisfying~\eqref{eq-Assumption-cor} is, also, sufficient for the norm-coercivity of $F_K^{\rm nat}$. In the sequel, we will illustrate this connection. 
\begin{lem}
Suppose that $K\subseteq \mathbb{R}^m$ is closed convex and $F:K\to \mathbb{R}^m$ is Lipschitz continuous. Moreover, assume that the normal mapping $F_K^{\rm nor}$ is norm-coercive on $\mathbb{R}^m$. Then, the natural mapping $F_K^{\rm nat}$ is norm-coercive on $K$.
\end{lem}
\begin{proof}
The proof is by contradiction. Let us assume there is a sequence $\{x^k\}\subset K$ such that $\lim_{k\to \infty}\|x^k\|=\infty$ and there is an $M>0$ such that $\|F_K^{\rm nat}(x^k)\|\leq M$ for every $k\geq 0$. By the projection non-expansiveness property, $x_k\in K$ for all $k$, and the definition of $F_K^{\rm nat}$, it follows that 
    \begin{equation} \label{eq-coercivity-natural map-auxiliary}
        \|x^k-\Pi_K[x^k-F(x^k)]\|\leq M\qquad\hbox{for every $k$}.
    \end{equation}
Consider the sequence $\{v^k\}$ with $v^k=x^k-F(x^k)$ for all $k$. If $\{v_k\}$ is bounded, then the sequence $\{\Pi_k[v^k]\}$ is also bounded. By~\eqref{eq-coercivity-natural map-auxiliary}, we get
$\|x^k\|\le \|x^k-\Pi_K[v^k]\|+\|\Pi_K[v^k]\|,$
which implies that $\{x^k\}$ is also bounded -- a contradiction. Thus, $\{v^k\}$ must be unbounded.

By the Lipschitz continuity of $F$ over $K$, from relation~\eqref{eq-coercivity-natural map-auxiliary}, it follows that  
        \[\|F(x^k)-F(\Pi_K[x^k-F(x^k)])\|\leq M'\qquad\hbox{for all $k$}.\]
    where $M'=LM$ with $L$ being the Lipschitz continuity constant of $F$.
    Using the preceding relation and~\eqref{eq-coercivity-natural map-auxiliary}, and the norm-triangle inequality, we obtain for every $k$,
\begin{align*}
   &\|x^k-F(x^k)-\Pi_K[x^k-F(x^k)]+F(\Pi_K[x^k-F(x^k)])\|\leq M+M'.
\end{align*}
    Since $v^k=x^k-F(x^k)$, it follows that
    for all $k$,
    \[\|v^k-\Pi_K[v^k]+F(\Pi_K[v^k])\| 
    \leq M+M',\]
    implying by the definition of the normal mapping that 
    $F_K^{\rm nor}(v^k)\le M+M'$ for all $k$, which contradicts the norm-coercivity of $F_K^{\rm nor}$.

\end{proof}

\subsubsection{Degree Theoretic Approach}
Here, we introduce degree theoretic notions that we use to study the existence of a solution to a VI. We show the existence of solutions based on the nearness of the VI mapping to another mapping for which a solution exists. We denote tha closure and boundary of a set by $\hbox{cl}(\cdot)$ and $\hbox{bd}(\cdot)$, respectively. 
Let $\Gamma$ be the collection of triples $(\varphi,\Omega,p)$, where $\varphi:\hbox{cl}(\Omega)\to \mathbb{R}^m$ is a continuous mapping, 
$\Omega\subset\mathbb{R}^m$ is a bounded open subset, and $p$ is a vector satisfying $p \notin \varphi\left(\hbox{bd}(\Omega)\right)$. For the collection $\Gamma$, we consider an integer-valued  function
$\hbox{deg}(\varphi,\Omega,p)$, i.e., $\hbox{deg}: \Gamma\to\mathbb{N}$, for which we have the following definition. 
 \begin{defn}[Definition 2.1.1 \cite{facchinei2003finite}] \label{Def-Degree}
     The function $\hbox{deg}$ is a (topological) degree if the following three axioms are satisfied:\\
 \noindent
 (Axiom 1) $\hbox{deg}(\mathbb{I},\Omega,p)=1$ if $p\in \Omega$.\\
 \noindent
 (Axiom 2) $\hbox{deg}(\varphi,\Omega,p)=\hbox{deg}(\varphi,\Omega_1,p)+\hbox{deg}(\varphi,\Omega_2,p)$ if $\Omega_1$ and $\Omega_2$ are two disjoint open subsets of $\Omega$ and $p\notin \varphi\left(\hbox{cl}(\Omega)\setminus(\Omega_1 \cup \Omega_2)\right)$.\\
\noindent (Axiom 3) $\hbox{deg}\left(H(\cdot,t),\Omega,p(t)\right)$ is independent of $t \in [0,1]$ for any two continuous functions $H : \hbox{cl}(\Omega)\times [0,1] \to \mathbb{R}^m$ and $p : [0,1] \to \mathbb{R}^m$ such that
         $$p(t)\notin H(\hbox{bd}(\Omega),t)\qquad\hbox{for all } t\in [0,1].$$

      We say that $\hbox{deg}(\varphi,\Omega,p)$ is the degree of $\varphi$ at the point $p$ relative to $\Omega$. If $p = 0$, we simply write $\hbox{deg}(\varphi,\Omega)$. 
 \end{defn}
 The following lemma provides the fundamental property of the degree function, which plays a key role in many existing results on solutions of variational inequalities.
 \begin{lem}[Theorem 2.1.2 \cite{facchinei2003finite}]\label{Lem-Degree Theorem}
 Let $(\varphi,\Omega,p)\in\Gamma$.
 If $\hbox{deg}(\varphi,\Omega,p)\neq 0$, then 
 $\varphi(\bar{x})=p$ for 
 some $\bar{x} \in \Omega$. Conversely, if $p \notin \hbox{cl}\left(\varphi(\Omega)\right)$, then $\hbox{deg}(\varphi,\Omega,p)= 0$.
 \end{lem}
 
 It is known that, when two mappings are close enough, they have the same degree, as seen in the next lemma.
\begin{lem}[Proposition 2.1.3 c) \cite{facchinei2003finite}]\label{Lem-Degree Nearness} 
Let $(\varphi,\Omega,p)\in\Gamma$, and let $\psi:\hbox{cl}(\Omega)\to\mathbb{R}^m$ be a continuous mapping.
Then, $\hbox{deg}(\varphi,\Omega,p)=\hbox{deg}(\psi,\Omega,p)$ if 
$$ \max_{x\in \hbox{cl}(\Omega) }\|\varphi(x)-\psi(x)\|_{\infty}<\inf_{z\in \varphi(\hbox{bd}(\Omega))}\|p-z\|_\infty.$$
Moreover, we have
$$ \max_{x\in \hbox{cl}(\Omega) }\|\varphi(x)-\psi(x)\|<\frac{1}{7}\inf_{z\in \varphi(\hbox{bd}(\Omega))}\|p-z\|.$$
 \end{lem}

The next lemma provides a relation between the natural mapping $\varphi_K^{\rm nat}$ a given $\xi$-monotone mapping $\varphi$. 

\begin{lem}[Theorem 2.3.3 c) \cite{facchinei2003finite}]\label{Lem-property-xi monotone mapping}
    Let us assume the set $K$ is closed convex and $\varphi:K\to \mathbb{R}^m$ be defined, Lipschitz continuous, and $\xi$-monotone on $\mathbb{R}^m$ for some $\xi>1$. Then for every $x\in \mathbb{R}^m$ we have 
    $$\frac{c_\varphi}{L_\varphi}\|x-x^*\|^{\xi-1}\leq\|\varphi_K^{\rm nat}(x)\|,$$
    where, $x^*$ is the unique solution of VI$(K,\varphi)$ and $L_\varphi$ and $c_\varphi$ are the Lipshitz constant and ${\xi}$-monotonicity constant of mapping $\varphi$, respectively.
\end{lem}

Using Lemmas~\ref{Lem-Degree Theorem}--\ref{Lem-property-xi monotone mapping}, we have 
 the following result for the existence of solutions to a VI.
 In what follows, we use notation $\hbox{dist}(x,S):=\inf_{y\in S}\|x-y\|.$
 \begin{thm}\label{Thm-Degree theorem main conclusion}
     Let $K\subseteq \mathbb{R}^m$ be a closed convex set and  $F:\Rem\to\mathbb{R}^m$ be a continuous mapping. Let the mapping $\varphi:\Rem\to\mathbb{R}^m$ be $\xi$-monotone with constant $c_\varphi$ and Lipschitz continuous with constant $L_\varphi$. Also, assume that for some $l\in\mathbb{R}^+$,
     we have $$\max_{x\in \mathbb{R}^m }\|\varphi_K^{\rm nat}(x)-F_K^{\rm nat}(x)\|<l.$$ Then, $\hbox{SOL}(K,F)$ is nonempty.
 \end{thm}
 \begin{proof}
 Since $\varphi$ is $\xi$-monotone, the VI$(K,\varphi)$ has a unique solution $x^*\in K$ (\cite{facchinei2003finite}, Theorem~2.3.3).
     By the $\xi$-monotonicity of $\varphi$, with $\xi>1$,  and Lemma~\ref{Lem-property-xi monotone mapping}, we have $\lim_{\|x\|\to\infty} \|\varphi_K^{\rm nat}(x)\|=+\infty$. 
     Thus, $\|\varphi_K^{\rm nat}(x)\|\geq 7l$ for some $r>0$ large enough so that 
     $\|x\|\geq r>\|x^*\|$. 
     
     Let $B_r(0)\subset \mathbb{R}^m$ be an open ball centered at the origin with the radius $r>0$, i.e., 
     $B_r(0):=\{y\in \mathbb{R}^m\mid \|x\|<r\}$.
     The boundary of this ball is $\hbox{bd}\big(B_r(0)\big):=\{x\in \mathbb{R}^m\mid |\|x\|=r\}$. 
     Since $\|\varphi_K^{\rm nat}(x)\|\geq 7l$ for all $x$ with $\|x\|\ge r$, it follows that $\hbox{dist}\left(0,\varphi_K^{\rm nat}(\hbox{bd}(B_r(0))\right)\geq 7l$. Therefore, by the given condition in the theorem, we have 
     \begin{align*}
         \max_{x\in \hbox{cl}\big(B_r(0)\big) }\|\varphi_K^{\rm nat}(x)-F_K^{\rm nat}(x)\|<l \leq \frac{1}{7}\hbox{dist}\left(0,\varphi_K^{\rm nat}\left(\hbox{bd}\big(B_r(0)\big)\right)\right).
     \end{align*}
     By Lemma~\ref{Lem-Degree Nearness}, $\hbox{deg}(\varphi_K^{\rm nat},B_r(0))=\hbox{deg}(F_K^{\rm nat},B_r(0))$. 
     
    Consider the mapping $\phi_t(x)=t\varphi(x) +(1-t)x$ for any $t\in [0,1]$. We have for every $x,y\in \Rem$,
     \begin{align}\label{eq-ksi_monotone_Phi}
         \langle t\varphi(x)+(1-t)x-\left(t\varphi(y)+(1-t)y\right),x-y\rangle\geq tc_\varphi\|x-y\|^\xi+(1-t)\|x-y\|^2.
     \end{align}
     Thus, the mapping $\phi_t$ is $\xi$- monotone for every $t\in(0,1]$ and strongly monotone for $t=0$. By Theorem~2.3.3 in~\cite{facchinei2003finite}, the VI$(K,\phi_t)$ has a unique solution $\bar{x}_t$ for every $t\in[0,1]$,
     which satisfies: for each $t\in[0,1]$,
 \begin{equation}\label{Eq_Sol_phi_t}
         \langle\phi_t(\bar{x}_t),x-\bar{x}_t \rangle\geq 0\qquad\hbox{for all $x\in K$}.
     \end{equation}
     Let $C_0$ be the set of solutions for the VI$(K,\phi_t)$ for $t\in[0,1]$, i.e.,
     $C_0:=\{\bar{x}_t \mid t\in[0,1]\}.$
    The mapping $\phi_t$ is continuous in $t$, so the set $C_0\subset K$ is closed.
     Letting  $x=x^{\rm ref}$ for some $x^{\rm ref}\in K\setminus C_0$ and $y=\bar{x}_t$ in~\eqref{eq-ksi_monotone_Phi}, and combining it with~\eqref{Eq_Sol_phi_t}, we have for every $t\in[0,1]$,
     \begin{align}\label{eq-ksi_monotone_Phi-2}
         \langle t\varphi(x^{\rm ref})+(1-t)x^{\rm ref},x^{\rm ref}-\bar{x}_t\rangle \geq tc_\varphi\|x^{\rm ref}-\bar{x}_t\|^\xi+(1-t)\|x^{\rm ref}-\bar{x}_t\|^2.
     \end{align}
     By the Cauchy Schwarz inequality, we further have 
     \begin{align}\label{eq-ksi_monotone_Phi-3}
         \|t\varphi(x^{\rm ref})+(1-t)x^{\rm ref}\|\|x^{\rm ref}-\bar{x}_t\|\geq tc_\varphi\|x^{\rm ref}-\bar{x}_t\|^\xi+(1-t)\|x^{\rm ref}-\bar{x}_t\|^2.
     \end{align}
     Let us define 
     For all $t\in [0,1]$,  we have $\|x^{\rm ref}-\bar{x}_t\|\neq 0$, so by dividing both sides of~\eqref{eq-ksi_monotone_Phi-3} with $\|x^{\rm ref}-\bar{x}_t\|$ we obtain 
     \begin{align*}
         \|t\varphi(x^{\rm ref})+(1-t)x^{\rm ref}\|\geq tc_\varphi\|x^{\rm ref}-\bar{x}_t\|^{\xi-1}+(1-t)\|x^{\rm ref}-\bar{x}_t\|.
     \end{align*}
     Hence,  the set $C_0$ is bounded. For every $t\in[0,1]$, the mapping $\phi_t$ is $\xi$-monotone with the constant $tc_{\phi_t}$ and strongly monotone with constant $1-t$. Also, it is Lipschitz continuous with the constant $tL_{\varphi}+(1-t)$. Thus,  by Lemma~\ref{Lem-property-xi monotone mapping} we have for every $t\in[0,1]$ and all $x\in\Rem$,
     \begin{align*}
        \max \left\{\frac{tc_\varphi\|x-\bar{x}_t\|^{\xi-1}}{tL_{\varphi}+(1-t)},\frac{(1-t)\|x-\bar{x}_t\|^{2}}{tL_{\varphi}+(1-t)}\right\}\leq\|(\phi_t)_K^{\rm nat}(x)\|.
     \end{align*}
     Since $tL_{\varphi}+(1-t)\leq \max\{L_{\varphi},1\}$, we have for ll $t\in[0,1]$ and $x\in\Rem$,
     \begin{align}\label{eq-applying-lemma8-1}
       \frac{\max \{tc_\varphi\|x-\bar{x}_t\|^{\xi-1},(1-t)\|x-\bar{x}_t\|^{2}\}}{\max\{L_{\varphi},1\}}\leq\|(\phi_t)_K^{\rm nat}(x)\|.
     \end{align}
     Let $C:=C_0\cup \{x^{\rm ref}\}$ and note that 
     the set $C$ is bounded. Thus, there is $r'>0$ such that for all $x$ with $\|x\|>r'$ and all $t\in[0,1]$ with $\bar{x}_t\in C$, we have $\|x-\bar{x}_t\|>0$. Therefore, the left hand side of~\eqref{eq-applying-lemma8-1} is positive for all $x$ with $\|x\|>r'$,
     implying that $0\notin (\phi_t)_K^{\rm nat}\left({\rm bd}(B_{r'}(0))\right)$
     for all $t\in[0,1]$.
     Let $r^0=\max\{r',\|\Pi_K(0)\|\}$, and  assume that ${\rm deg}(\varphi^{\rm nat}_K,B_{r^0}(0))=0$. By the invariance property of the degree (Axiom~3 of Definition~\ref{Def-Degree}) for the homotopy $H(x,t):=(\phi_t)_K^{\rm nat}(x)$, we have for $t=0$,
     \[{\rm deg}(\mathbb{I}-\Pi_K[0],B_{r^0}(0))=0.\]
     Using Proposition~2.1.3(a) of \cite{facchinei2003finite}, we have \[{\rm deg}(\mathbb{I}-\Pi_K[0],B_{r^0}(0))={\rm deg}(\mathbb{I},B_{r^0}(0),\Pi_K[0]).\]
    Hence, ${\rm deg}(\mathbb{I},B_{r^0}(0),\Pi_K[0])=0$; however, this yields a contradiction with Axiom~1 of Definition~\ref{Def-Degree} since $\Pi_K[0]\in B_{r^0}(0)$. Thus, ${\rm deg}(\varphi_K^{\rm nat},B_{r^0}(0))\neq 0$. Considering $r'>r$ and Lemma~\ref{Lem-property-xi monotone mapping} we can conclude that ${\rm deg}(\varphi_K^{\rm nat},B_{r^0}(0))\neq 0$. In addition, using the fact that ${\rm deg}(\varphi_K^{\rm nat},B_{r^0}(0))={\rm deg}(F_K^{\rm nat},B_{r^0}(0))$ by Lemma~\ref{Lem-Degree Theorem}, there is  $\bar{x}\in \hbox{cl}(B_{r^0}(0))\subset \mathbb{R}^m$ such that $F_K^{\rm nat}(\bar{x})=0$. By Theorem~\ref{Thm-auxiliary-natural map solution}, the point $\bar x$ is a solution to VI$(K,F)$. 
     
 \end{proof}
\begin{rem}
    The condition on the natural maps of Theorem~\ref{Thm-Degree theorem main conclusion} will hold if for  some $l\in\mathbb{R}^+$,
    \[\max_{x\in \Rem} \|\varphi(x)-F(x)\|<l. \] 
\end{rem}
The preceding condition is helpful when we have no access to the mapping $F$ itself to check the conditions similar to condition~2.2.2 of \cite{facchinei2003finite}, but we are aware that the mapping is continuous and obtained by some bounded perturbation of a $\xi$-monotone mapping. It should also be noted that even when the form of $F$ is available, finding a reference point and $\xi$ satisfying the condition 2.2.2 in \cite{facchinei2003finite} can be a challenging task. Our Theorem~~\ref{Thm-Degree theorem main conclusion} may provide a useful alternative.

Next, we provide sufficient conditions for the existence of a strong Minty solution to VI$(K,F)$ (see Definition~\ref{Def-Strong-Minty}). 

\begin{thm}\label{Thm-Minty sol}
Let the set $K\subseteq\mathbb{R}^m$ be nonempty, closed, and convex, and let the mapping $\varphi:K\to\mathbb{R}^m$ be \sa{continuous and} strongly monotone with the strong monotonicity constant $\mu_{\varphi}>0$. Let $\Tilde{x}$ be the unique solution to the VI$(K,\varphi)$. 
Assume that $\|\varphi(x)-F(x)\|\le d\|x-\Tilde{x}\|$ for some $d <\mu_{\varphi}$and for all $x\in K$.
Then, $\Tilde{x}$ is a strong Minty solution to VI$(K,F)$.
\end{thm}
\begin{proof}
 The  VI$(K,\varphi)$  has a unique solution (Theorem~2.3.3 in \cite{facchinei2003finite}). By the strong monotonicity of $\varphi$, for the solution $\Tilde{x}$, we have
\begin{equation}\label{eq-var}
\langle\varphi(x), x-\Tilde{x}\rangle \geq \mu_{\varphi}\|x-\Tilde{x}\|^2\qquad\hbox{for all }x\in K.
\end{equation}
By the assumption that $\|\varphi(x)-F(x)\|\le d\|x-\Tilde{x}\|$ for all $x\in K$, it follows that
$$
\langle\varphi(x)-F(x), x-\Tilde{x}\rangle \leq d\|x-\Tilde{x}\|^2\qquad\hbox{for all }x\in K.
$$
Hence, we have $$(\mu_{\varphi}-d)\|x-\Tilde{x}\|^2\leq \langle F(x), x-\Tilde{x}\rangle,$$
implying that $\tilde{x}$ is a strong Minty solution to VI$(K,F)$.

\end{proof}

\section{Algorithms}\label{Sec-Algorithm analysis}

In this section, we consider Korpelevich~\cite{C4} (aka the extra-gradient) and Popov~\cite{popov1980modification} methods for solving a VI problem.
These methods have been extensively studied for monotone VIs (\cite{facchinei2003finite, yousefian2014optimal, farzad2017, farzad2018, beznosikov2023smooth}). In many applications, such as those arising in machine learning the mapping is not necessarily monotone and we focus on this case.
Under some conditions, we show that both Korpelevich and Popov methods converge sub-sequentially to a solution set
even for non-monotone VIs.

\begin{assum}\label{Assum-set-Lip}
    Let $K\subseteq\mathbb{R}^m$ be a nonempty closed convex set, and let the mapping  $F:K\rightarrow \mathbb{R}^m$  be Lipschitz continuous on $K$, i.e., 
    there exists a constant $L$ such that $\|F(x)-F(y)\|\leq L\|x-y\|$ for all $x,y\in K$.
\end{assum}

The Korpelevich method is given by: for all $k\ge0$,
\begin{align}\label{Alg-Korpelevich}
        &y^k=\Pi_K[x^k-\alpha F(x^k)],\nonumber \\
        &x^{k+1}=\Pi_K[x^k-\alpha F(y^k)],
    \end{align}
    where $\alpha>0$ is a stepsize, and $x^0,y^0\in K$ are arbitrary initial points.
The next theorem shows that, having a Minty solution to VI$(K,F)$,  the Korpelevich method generates a bounded sequence $\{x^k\}$ with accumulation points in the set SOL$(K,F)$, for a suitable selection of the stepize. 

\begin{thm}\label{Thm-extragradient main variant}
    Let Assumption~\ref{Assum-set-Lip} hold. Assume that the VI$(K,F)$ has a Minty solution.
    Then, the sequence $\{x^k\}$ generated by the Korpelevich method~\eqref{Alg-Korpelevich}, with the stepsize $0<\alpha<\frac{1}{L}$, is bounded and all of its accumulation points lie in the solution set SOL$(K,F)$. 
\end{thm}
\begin{proof}
     Using the properties of the projection, from the definition of the iterate $x^{k+1}$ we have that for all $k\ge0$ and any $x\in K$,
\begin{align}\label{eq-alg 1}
        \|x^{k+1}-x\|^2\leq&\|x^k-\alpha F(y^k)-x\|^2 -\|x^k-\alpha F(y^k)-x^{k+1}\|^2\cr
        =&\|x^{k}-x\|^2-\|x^{k}-x^{k+1}\|^2+2\alpha\langle F(y^k),x-x^{k+1} \rangle.
    \end{align}
    Letting $x=\Tilde{x}$, where
    $\Tilde{x}$ is a Minty solution, since $y^k\in K$, we have $\langle F(y^k),\Tilde{x}-y^{k} \rangle \leq 0$ for all $k\ge0$. Therefore,\begin{align}\label{eq-alg 2}
        \langle F(y^k),\Tilde{x}-x^{k+1} \rangle&=\langle F(y^k),\Tilde{x}-y^{k} \rangle 
        +\langle F(y^k),y^{k}-x^{k+1} \rangle \leq \langle F(y^k),y^{k}-x^{k+1} \rangle.
    \end{align}
    Using \eqref{eq-alg 2} in \eqref{eq-alg 1}, where $x=\tilde x$, we can write
     \begin{align}\label{eq-alg 3}
        \|x^{k+1}-\Tilde{x}\|^2\leq&\|x^{k}-\Tilde{x}\|^2-\|x^{k}-x^{k+1}\|^2 +2\alpha\langle F(y^k),y^{k}-x^{k+1} \rangle \nonumber \\
        =&\|x^{k}-\Tilde{x}\|^2-\|x^{k}-y^{k}\|^2-\|y^{k}-x^{k+1}\|^2 \nonumber \\
        &-2\langle x^{k}-y^{k},y^{k}-x^{k+1} \rangle+2\alpha\langle F(y^k),y^{k}-x^{k+1} \rangle \nonumber \\
        =&\|x^{k}-\Tilde{x}\|^2-\|x^{k}-y^{k}\|^2-\|y^{k}-x^{k+1}\|^2\nonumber \\
        &+2\langle x^{k}-\alpha F(y^k)-y^{k},x^{k+1}-y^{k} \rangle.
    \end{align}
We can estimate the last term in~\eqref{eq-alg 3} in the following form using the Cauchy inequality
\begin{align}\label{eq-alg 4}
        \langle x^{k}-\alpha F(y^k)-y^{k},x^{k+1}-y^{k} \rangle
        = &\langle x^{k}-\alpha F(x^k)-y^{k},x^{k+1}-y^{k} \rangle +\alpha\langle  F(x^k)-F(y^k),x^{k+1}-y^{k} \rangle \nonumber \\
        \leq &
        \alpha \langle F(x^k)-F(y^k),x^{k+1}-y^{k} \rangle \nonumber \\
        \leq &\alpha\|F(x^k)-F(y^k)\|\|x^{k+1}-y^{k}\|,
    \end{align}
    where the first inequality is obtained using the fact that $\langle x^{k}-\alpha F(x^k)-y^{k},x^{k+1}-y^{k} \rangle \le 0$ which follows from the projection inequality  $\langle z-\Pi_K[z],x-\Pi_K[z]\rangle\le 0$ for all $z\in\mathbb{R}^m$ and $x\in K$, the definition of $y^k$, and $x^{k+1}\in K$. 
Combining~\eqref{eq-alg 3} and \eqref{eq-alg 4}, and using the Lipschitz continuity of the mapping $F$,  we obtain 
\begin{align}\label{eq-alg 6}
        \|x^{k+1}-\Tilde{x}\|^2\leq&\|x^{k}-\Tilde{x}\|^2-\|x^{k}-y^{k}\|^2-\|y^{k}-x^{k+1}\|^2 +2\alpha L\|x^k-y^k\|\|x^{k+1}-y^{k}\| \nonumber \\
        \leq&\|x^{k}-\Tilde{x}\|^2-\|x^{k}-y^{k}\|^2-\|y^{k}-x^{k+1}\|^2 +\alpha^2 L^2\|x^k-y^k\|^2+\|x^{k+1}-y^{k}\|^2,
    \end{align}
    where in the last inequality in~\eqref{eq-alg 6} we use 
\[2\alpha  L\|x^k-y^k\|\|x^{k+1} - y^k\|
\le 
        \alpha^2 L^2\|x^k-y^k\|^2+\|x^{k+1} - y^k\|^2.\]
From~\eqref{eq-alg 6} it follows that for all $k\ge0$, 
\begin{align}
\label{eq-alg-main-eq-for convergence rate}
        \|x^{k+1}-\Tilde{x}\|^2&\leq\|x^{k}-\Tilde{x}\|^2-(1-\alpha^2 L^2)\|x^k-y^k\|^2.
    \end{align}
Therefore, for $\alpha<\frac{1}{L}$, we see that $\|x^k-y^k\|\to0$ and $\|x^{k}-\Tilde{x}\|^2$ converges, and, hence, $\{x^{k}\}_{k=0}^\infty$ is bounded. As a result, for every convergent subsequence $\{x^{k_i}\}_{i=1}^\infty$ with $\lim_{i\rightarrow\infty}x^{k_i}=\hat{x}$, the limit point $\hat x$ is in the set $K$ since $K$ is closed.
From the definition of $y^k$, we have that 
\[\|y^k-x^k\|=\|\Pi_K[x^k-\alpha F(x^k)]-x^k\|.\]
Since $\|x^k-y^k\|\to0$, it follows that 
$\|\Pi_K[x^k-\alpha F(x^k)]-x^k\|\to 0$.
Thus, for any convergent subsequence $\{x^{k_i}\}_{i=1}^\infty$ with
$\lim_{i\rightarrow\infty}x^{k_i}=\hat{x}$, it follows that $\|\Pi_K[x^k-\alpha F(x^k)]-x^k\|=0,$
implying that 
$\hat{x}=\Pi_K[\hat{x}-\alpha F(\hat{x})]$. 
Hence, $F_K^{\rm nat}(\hat x)=0$ and by Theorem~\ref{Thm-auxiliary-natural map solution}, we have that 
$\hat{x}\in {\rm SOL}(K,F)$. 

\end{proof}
We note that the assumption of Theorem~\ref{Thm-extragradient main variant} that VI$(K,F)$ has a Minty solution together with the assumption that the mapping $F$ is Lipschitz continuous implies that VI$(K,F)$ has a solution
by Lemma~\ref{Lem-Minty Lemma}(a).

Next, we discuss the Popov method, which is given by: for all $k\ge0$,
\begin{align}\label{Alg-Popov}
        &x^{k+1}=\Pi_K[x^k-\alpha F(y^k)],\nonumber \\
        &y^{k+1}=\Pi_K[x^{k+1}-\alpha F(y^k)],
    \end{align}
    where $\alpha>0$ is a stepsize, and $x^0,y^0\in K$ are arbitrary initial points.
\begin{thm}\label{Thm-Popov main variant}
    Let Assumption~\ref{Assum-set-Lip} hold. Assume that the VI$(K,F)$ has a Minty solution. 
    Then, the sequence $\{x^k\}$ generated by the Popov method~\eqref{Alg-Popov}, with the stepsize $0<\alpha<\frac{1}{3L}$, is bounded and all of its accumulation points lie in the solution set SOL$(K,F)$.
\end{thm}
\begin{proof}
    Let $\Tilde{x}\in K$ be a Minty solution to VI$(K,F)$.
     Using projection property  $\langle z-\Pi_K[z],x-\Pi_K[z]\rangle\le 0$ for all $z\in\mathbb{R}^m$ and $x\in K$, and the definition of the iterate $x^{k+1}$, we have for $\|x^{k+1}-\tilde{x}\|^2$
    for all $k\ge0$,
\begin{align}\label{eq-alg 1-Popov}
        \|x^{k+1}-\tilde{x}\|^2\leq&\|x^k-\alpha F(y^k)-\tilde{x}\|^2-\|x^k-\alpha F(y^k)-x^{k+1}\|^2\cr
        =&\|x^{k}-x\|^2-\|x^{k}-x^{k+1}\|^2 +2\alpha\langle F(y^k),\tilde{x}-x^{k+1} \rangle.
    \end{align}
    Since $y^k\in K$ and $\Tilde{x}$ is a Minty solution, we have $\langle F(y^k),\Tilde{x}-y^{k} \rangle \leq 0$ for all $k\ge0$. Therefore,\begin{align}\label{eq-alg 2-Popov}
        \langle F(y^k),\Tilde{x}-x^{k+1} \rangle&=\langle F(y^k),\Tilde{x}-y^{k} \rangle 
        +\langle F(y^k),y^{k}-x^{k+1} \rangle \leq \langle F(y^k),y^{k}-x^{k+1} \rangle.
    \end{align}
    Using \eqref{eq-alg 2-Popov} in \eqref{eq-alg 1-Popov} we can write
     \begin{align*}
        \|x^{k+1}-\Tilde{x}\|^2\leq&\|x^{k}-\Tilde{x}\|^2-\|x^{k}-x^{k+1}\|^2 +2\alpha\langle F(y^k),y^{k}-x^{k+1} \rangle \nonumber \\
        =&\|x^{k}-\Tilde{x}\|^2-\|x^{k}-y^{k}\|^2-\|y^{k}-x^{k+1}\|^2 \nonumber \\
        &+2\alpha\langle F(y^{k-1})-F(y^{k}),x^{k+1}-y^{k} \rangle \cr &+2\langle x^{k}-\alpha F(y^{k-1})-y^{k},x^{k+1}-y^{k}\rangle. 
    \end{align*}
By the projection property $\langle x^{k}-\alpha F(y^{k-1})-y^{k},x^{k+1}-y^{k} \rangle \le 0$, the last term in the preceding relation is non-positive. 
Using the Cauchy-Schwarz inequality and the Lipschitz continuity of $F$,  we obtain for all $k\ge0$,
\begin{align*}
        \|x^{k+1}-\Tilde{x}\|^2\leq&\|x^{k}-\Tilde{x}\|^2-\|x^{k}-y^{k}\|^2-\|y^{k}-x^{k+1}\|^2 +2\alpha L(\|y^{k-1}-x^k\|+\|x^k-y^k\|)\|x^{k+1}-y^{k}\| \nonumber \\
        \leq&\|x^{k}-\Tilde{x}\|^2 -(1-2\alpha L)\|x^{k+1}-y^k\|^2-(1-\alpha L)\|y^{k}-x^{k}\|^2+\alpha L\|x^k-y^{k-1}\|^2.
    \end{align*}
From the preceding inequality, by summing the relations over $k=N,\ldots,M$ 
for some $M>N\ge0$,
it follows that 
\begin{align}
\label{eq-alg-main-eq-for convergence rate-popov}
        \|x^{M+1}-\Tilde{x}\|^2&\leq\|x^{N}-\Tilde{x}\|^2+\alpha L\|x^N-y^{N-1}\|^2\nonumber \\
        & -(1-\alpha L)\sum_{k=N}^{M}\|x^k-y^k\|^2\nonumber \\
        & -(1-3\alpha L)\sum_{k=N}^{M-1}\|x^{k+1}-y^k\|^2\nonumber \\
        & -(1-2\alpha L)\|x^{M+1}-y^M\|^2.
    \end{align}
    Since  $\alpha <\frac{1}{3L}$, 
    the sequence $\{x^k\}$ is bounded. Thus, for any convergent subsequence $x^{k_t}\to x^*$, from \eqref{eq-alg-main-eq-for convergence rate-popov} we have that $\|x^{k+1}-y^k\|$ and $\|x^k-y^k\|$ converge to zero. Therefore, as $k_t\to \infty$ we have $\Pi_K[x^*-\alpha F(x^*)]=x^*$ implying that $F_K^{\rm nat}(x^*)=0$. By Theorem~\ref{Thm-auxiliary-natural map solution},  $x^*\in {\rm SOL}(K,F)$. 
    
    \end{proof}

 \an{In Tables~\ref{Table-EG} and~\ref{Table-Popov}, we compare the assumptions in some related papers and those of this paper that were used in establishing the convergence properties of Korpelevich and Popov methods, respectively. Some of the related papers treat stochastic VIs, so we compare with their deterministic counterparts. The signs 
 $+$ and $-$ are used to indicate that an assumption is used or not used, respectively. The use of ``solutions" Tables~\ref{Table-EG} and~\ref{Table-Popov} indicates that the existence of such a solution has been assumed.} 
 
 
\begin{table}[h!]
\centering
\caption{Assumptions for (deterministic) Korpelevich method}
\begin{tabular}{lcccc}
\toprule
\textbf{} & \textbf{\cite{hsieh2020explore}} & \textbf{\cite{vankov2024generalized}} & \textbf{\cite{choudhury2023single}} & \textbf{Our} \\
\midrule
Constrained VI & - & + & - & + \\
Minty Solution & + & + & - & + \\
Strong Minty Solution & - & + & - & - \\
Lipschitz Continuity of $F$ & + & - & + & + \\
Fixed Step Size & - & - & + & + \\
\bottomrule
\label{Table-EG} 
\end{tabular}
\end{table}
\begin{table}[h!] 
\centering
\caption{Assumptions for (deterministic) Popov method}
\begin{tabular}{lccc} 
\toprule
\textbf{} & \textbf{\cite{vankov2024generalized}} & \textbf{\cite{vankov2023last}} & \textbf{Our} \\
\midrule
Constrained VI & + & + & +  \\
Minty Solution & + & + & +\\
Strong Minty Solution & + & + & - \\
Lipschitz Continuity of $F$ & - & + & +  \\
Fixed Step Size & - & - & +  \\
\bottomrule
\label{Table-Popov} 
\end{tabular}
\end{table}

\section{Conclusions}\label{Sec-conclusion}
In this paper, we studied unconstrained non-monotone VIs through the inverse mapping theorem and have obtained some conditions for the existence of solutions to such VIs. 
We then extended these results for the case of constrained non-monotone VIs where the domain is nonempty closed convex but not necessarily compact. We have also developed new solution existence results using the degree theory.
Moreover, we derived some conditions that guarantee the existence of a Minty solution given that at least one solution exists for a related VI. Finally, we showed that the Korpelevich and Popov methods produce iterates with accumulation points in the solution set of the VI problem with a non-monotone mapping, assuming that a Minty solution exists.

\bibliographystyle{plain}
\bibliography{arr}

\begin{thebibliography}{10}

\bibitem{arefizadeh2024non}
Sina Arefizadeh and Angelia Nedi{\'c}.
\newblock Non-monotone variational inequalities.
\newblock In {\em 2024 60th Annual Allerton Conference on Communication, Control, and Computing}, pages 01--07. IEEE, 2024.

\bibitem{bell1965gershgorin}
Howard~E Bell.
\newblock Gershgorin's theorem and the zeros of polynomials.
\newblock {\em The American Mathematical Monthly}, 72(3):292--295, 1965.

\bibitem{beznosikov2023smooth}
Aleksandr Beznosikov, Boris Polyak, Eduard Gorbunov, Dmitry Kovalev, and Alexander Gasnikov.
\newblock Smooth monotone stochastic variational inequalities and saddle point problems: A survey.
\newblock {\em European Mathematical Society Magazine}, (127):15--28, 2023.

\bibitem{C16}
Brian Bullins and Kevin~A Lai.
\newblock Higher-order methods for convex-concave min-max optimization and monotone variational inequalities.
\newblock {\em SIAM Journal on Optimization}, 32(3):2208--2229, 2022.

\bibitem{choudhury2023single}
Sayantan Choudhury, Eduard Gorbunov, and Nicolas Loizou.
\newblock Single-call stochastic extragradient methods for structured non-monotone variational inequalities: Improved analysis under weaker conditions.
\newblock {\em Advances in Neural Information Processing Systems}, 36:64918--64956, 2023.

\bibitem{clarke1976inverse}
Francis Clarke.
\newblock On the inverse function theorem.
\newblock {\em Pacific Journal of Mathematics}, 64(1):97--102, 1976.

\bibitem{crespi2005existence}
Giovanni~P Crespi, Ivan Ginchev, and Matteo Rocca.
\newblock Existence of solutions and star-shapedness in minty variational inequalities.
\newblock {\em Journal of Global Optimization}, 32:485--494, 2005.

\bibitem{C20}
Francisco Facchinei and Christian Kanzow.
\newblock Generalized nash equilibrium problems.
\newblock {\em Annals of Operations Research}, 175(1):177--211, 2010.

\bibitem{facchinei2003finite}
Francisco Facchinei and Jong-Shi Pang.
\newblock {\em Finite-dimensional variational inequalities and complementarity problems}.
\newblock Springer, 2003.

\bibitem{hsieh2020explore}
Yu-Guan Hsieh, Franck Iutzeler, J{\'e}r{\^o}me Malick, and Panayotis Mertikopoulos.
\newblock Explore aggressively, update conservatively: Stochastic extragradient methods with variable stepsize scaling.
\newblock {\em Advances in Neural Information Processing Systems}, 33:16223--16234, 2020.

\bibitem{C2}
Kevin Huang and Shuzhong Zhang.
\newblock Beyond monotone variational inequalities: Solution methods and iteration complexities.
\newblock {\em arXiv preprint arXiv:2304.04153}, 2023.

\bibitem{KSbook}
David Kinderlehrer and Guido Stampacchia.
\newblock {\em An introduction to variational inequalities and their applications}.
\newblock SIAM, 2000.

\bibitem{C4}
Galina~M Korpelevich.
\newblock The extragradient method for finding saddle points and other problems.
\newblock {\em Matecon}, 12:747--756, 1976.

\bibitem{lebl2018introduction}
Jir{\'\i} Lebl.
\newblock Introduction to real analysis, volume i, 2018.

\bibitem{C1}
Andrei Ludu.
\newblock {\em Nonlinear waves and solitons on contours and closed surfaces}.
\newblock Springer Science \& Business Media, 2012.

\bibitem{C10}
Bernard Martinet.
\newblock Regularization of variational inequalities using successive approximations, 1970.

\bibitem{C15}
Renato~DC Monteiro and Benar~F Svaiter.
\newblock Iteration-complexity of a newton proximal extragradient method for monotone variational inequalities and inclusion problems.
\newblock {\em SIAM Journal on Optimization}, 22(3):914--935, 2012.

\bibitem{C14}
Renato~DC Monteiro and Benar~Fux Svaiter.
\newblock On the complexity of the hybrid proximal extragradient method for the iterates and the ergodic mean.
\newblock {\em SIAM Journal on Optimization}, 20(6):2755--2787, 2010.

\bibitem{C12}
Arkadi Nemirovski.
\newblock Prox-method with rate of convergence o (1/t) for variational inequalities with lipschitz continuous monotone operators and smooth convex-concave saddle point problems.
\newblock {\em SIAM Journal on Optimization}, 15(1):229--251, 2004.

\bibitem{C13}
Yurii Nesterov.
\newblock Dual extrapolation and its applications to solving variational inequalities and related problems.
\newblock {\em Mathematical Programming}, 109(2):319--344, 2007.

\bibitem{noor2020new}
Muhammad~Aslam Noor, Khalida~Inayat Noor, and Michael~Th Rassias.
\newblock New trends in general variational inequalities.
\newblock {\em Acta Applicandae Mathematicae}, 170:981--1064, 2020.

\bibitem{C17}
Petr Ostroukhov, Rinat Kamalov, Pavel Dvurechensky, and Alexander Gasnikov.
\newblock Tensor methods for strongly convex strongly concave saddle point problems and strongly monotone variational inequalities.
\newblock {\em arXiv preprint arXiv:2012.15595}, 2020.

\bibitem{popov1980modification}
Leonid~Denisovich Popov.
\newblock A modification of the arrow-hurwitz method of search for saddle points.
\newblock {\em Mat. Zametki}, 28(5):777--784, 1980.

\bibitem{C9}
Mo{\"\i}se Sibony.
\newblock M{\'e}thodes it{\'e}ratives pour les {\'e}quations et in{\'e}quations aux d{\'e}riv{\'e}es partielles non lin{\'e}aires de type monotone.
\newblock {\em Calcolo}, 7:65--183, 1970.

\bibitem{C11}
Paul Tseng.
\newblock A modified forward-backward splitting method for maximal monotone mappings.
\newblock {\em SIAM Journal on Control and Optimization}, 38(2):431--446, 2000.

\bibitem{vankov2023last}
Daniil Vankov, Angelia Nedi\'c, and Lalitha Sankar.
\newblock Last iterate convergence of popov method for non-monotone stochastic variational inequalities.
\newblock {\em arXiv preprint arXiv:2310.16910}, 2023.

\bibitem{vankov2024generalized}
Daniil Vankov, Angelia Nedi\'c, and Lalitha Sankar.
\newblock Generalized smooth variational inequalities: methods with adaptive stepsizes.
\newblock In {\em Forty-first International Conference on Machine Learning}, 2024.

\bibitem{farzad2017}
F.~Yousefian, A.~Nedi\'c, and U.~V. Shanbhag.
\newblock On smoothing, regularization and averaging in stochastic approximation methods for stochastic variational inequalities.
\newblock {\em Mathematical Programming, Series B}, 165(1):391--431, 2017.

\bibitem{farzad2018}
F.~Yousefian, A.~Nedi\'c, and U.~V. Shanbhag.
\newblock On stochastic mirror-prox algorithms for stochastic cartesian variational inequalities: randomized block coordinate, and optimal averaging schemes.
\newblock {\em Set-Valued and Variational Analysis}, 26(4):789--819, 2018.

\bibitem{yousefian2014optimal}
Farzad Yousefian, Angelia Nedi{\'c}, and Uday~V Shanbhag.
\newblock Optimal robust smoothing extragradient algorithms for stochastic variational inequality problems.
\newblock In {\em 53rd IEEE Conference on Decision and Control}, pages 5831--5836. IEEE, 2014.

\end{thebibliography}
\end{document}